\documentclass[a4paper,12pt]{article}
\usepackage{mathrsfs}

\usepackage{amsfonts}
\usepackage{amscd,color}
\usepackage{amsmath,amsfonts,amssymb,amscd}
\usepackage{indentfirst,graphicx,epsfig}
\input{epsf}

\usepackage{epstopdf}
\usepackage{caption}

\usepackage[dvipsnames]{xcolor}
\usepackage[T1]{fontenc}
\usepackage[utf8]{inputenc}
\usepackage[english]{babel}
\usepackage{amsmath,amsthm,enumerate}

\usepackage[colorlinks=false, linktocpage=true]{hyperref}

\usepackage{url}
\usepackage{stmaryrd}
\usepackage{tikz,pgf}
\usepackage{subfig}
\usepackage{cancel}
\usetikzlibrary{positioning}
\usepackage{array}
\usepackage{makecell}
\usepackage{tabularx}
\usepackage{algorithm}
\usepackage{algorithmic}
\usepackage{comment}

\usetikzlibrary{decorations.markings}
\usetikzlibrary{decorations.pathmorphing}
\usetikzlibrary{patterns}
\tikzset{->-/.style={decoration={
  markings,
  mark=at position .5 with {\arrow[scale=0.8]{>}}},postaction={decorate}}}
\tikzset{snake it/.style={decorate, decoration={snake, amplitude=.4mm, segment length=2mm}}}

\newtheorem{thm}{Theorem}[section]
\newtheorem{df}[thm]{Definition}
\newtheorem{problem}[thm]{Problem}

\newtheorem{lem}[thm]{Lemma}

\newtheorem{cor}[thm]{Corollary}

\newcommand{\ml}{l\kern-0.55mm\char39\kern-0.3mm}

\title{\textbf{Cooperative coloring of some graph families\footnote{Supported by the Natural Science Basic Research Program of Shaanxi (Nos. 2022JQ-009, 2023-JC-YB-001, 2023-JC-YB-054)
and the Fundamental Research Funds for the Central Universities (No. XJS220705).}}}
\author{Xuqing Bai, \quad Bi Li, \quad Chuandong Xu\thanks{Corresponding author.}, \quad Xin Zhang\\
{\small School of Mathematics and Statistics, Xidian University, Xi'an, 710071, China}\\
{\small $\{$baixuqing, libi, xuchuandong, xzhang$\}$@xidian.edu.cn}}
\date{}

\begin{document}

\maketitle

\begin{abstract}
    In a family ${G_1, G_2, \ldots, G_m}$ of graphs sharing the same vertex set $V$, a cooperative coloring involves selecting one independent set $I_i$ from $G_i$ for each $i\in \{1,2,\ldots,m\}$ such that $\bigcup_{i=1}^m I_i = V$. For a graph class $\mathcal{G}$, let $m_{\mathcal{G}}(d)$ denote the minimum $m$ required to ensure that any graph family ${G_1, G_2, \ldots, G_m}$ on the same vertex set, where $G_i\in\mathcal{G}$ and $\Delta(G_i)\leq d$ for each $i\in \{1,2,\ldots,m\}$, admits a cooperative coloring. For the graph classes $\mathcal{T}$ (trees) and $\mathcal{W}$ (wheels), we find that $m_\mathcal{T}(3)=4$ and $m_\mathcal{W}(4)=5$. 
    Also, we prove that $m_{\mathcal{B}^*}(d)=O(\log_2 d)$ and $m_{\mathcal{L}}(d)=O\left(\frac{\log d}{\log\log d}\right)$, where $\mathcal{B}^*$ represents the class of graphs whose components are balanced complete bipartite graphs, and $\mathcal{L}$ represents the class of graphs whose components are generalized theta graphs.
\end{abstract}



\section{Introduction}

All graphs discussed in this paper are finite, undirected, and loopless. For notation or terminology not explicitly defined here, we follow those used in Bondy \cite{BM0}.

In a family $G_1,G_2,\ldots, G_m$ (not necessarily distinct) of graphs sharing the same vertex set $V$, a \emph{cooperative coloring} is a selection of one independent set $I_i$ from $G_i$ for each $i\in [m]$ such that $\bigcup^m_{i=1}I_i=V$, where $[m]=\{1,2,\ldots,m\}$. Let $G$ be an edge-colored multigraph with an edge coloring $\phi:E(G)\to [m]$. An \emph{adapted coloring} on $G$ is a vertex coloring $\sigma:V(G)\to [m]$ in which no edge is assigned the same color as both of its endpoints. Formally, this means that for every edge $uv$ in $E(G)$, we have $\lnot \big(\phi(uv)=\sigma(u)=\sigma(v)\big)$. It is worth noting that only the colors used for the edges are available for use in an adapted coloring of $G$. For each $i\in [m]$, let $G_i$ denote the graph with vertex set $V(G)$ and edge set $\phi^{-1}(i)$. It can be verified that a cooperative coloring of the graph family $G_1,G_2,\ldots, G_m$ is equivalent to an adapted coloring of $G$. The concept of adapted coloring was initially introduced by Kostochka and Zhu \cite{KZ} and has subsequently been extensively studied \cite{EMZ, HZ, M, MT}.

Let $m(d)$ be the minimum $m$ such that for any family $G_1,G_2,\ldots,G_m$ of graphs on the same vertex set, where $\Delta(G_i)\leq d$ for all $i\in [m]$, there exists a cooperative coloring. For a graph $G$ with $\Delta(G)\leq d$, it is possible to partition its vertex set greedily into $d+1$ independent sets. Consequently, cooperative coloring can be achieved for $d+1$ identical copies of $G$. However, there exist $d+1$ graphs, all with maximum degree $d$, sharing the same vertex set, yet they do not have a cooperative coloring (Aharoni et al.,\cite{AHHS}). Therefore, we have the lower bound $m(d)\geq d+2$. On the other hand, Haxell \cite{H} proved that $m(d)$ does not exceed $2d$. When all the graphs within the graph family are locally sparse, Loh and Sudakov \cite{LS} showed that the upper bound $2d$ can be improved to $d+o(d)$.

The concept of cooperative coloring can be generalized by relaxing the constraint that all graphs must share the same vertex set. When dealing with a family ${G_1,G_2,\ldots, G_m}$ of graphs with vertex sets $V_1, V_2,\ldots, V_m$ (not necessarily the same), a \emph{cooperative list coloring} involves selecting one independent set from $G_i$ for each $i\in [m]$ in such a way that their union covers the vertex set $V=\bigcup^m_{i=1}V_i$. Bradshaw \cite{B} introduced the notation $l(d)$ to represent the minimum value of $l$ such that every family ${G_1,G_2,\ldots,G_m}$ of graphs with $\Delta(G_i)\leq d$ for $i\in [m]$, and where each vertex $v$ in $\bigcup^m_{i=1}V(G_i)$ belongs to at least $l$ graphs in this family, admits a cooperative list coloring. In this definition, the graph family must consist of at least $l$ graphs since each vertex must be part of at least $l$ graphs. Bradshaw \cite{B} summarized results about $m(d)$ and $l(d)$ with the following inequality:
\begin{equation}\label{ineq1}
    d+2\leq m(d)\leq l(d) \leq 2d.
\end{equation}

Even for $d=3$, the precise value of $m(d)$ remains unknown. To refine understanding of this problem, researchers have delved into its study within specific graph classes. The first of the following definitions was introduced by Aharoni et al.~\cite{ABCHJ}, while the latter was proposed by Bradshaw \cite{B}.


\begin{df}{\upshape \cite{ABCHJ}}
For a graph class $\mathcal{G}$, let $m_{\mathcal{G}}(d)$ be the minimum value of $m$ such that any family $G_1,G_2,\ldots,G_m$ of graphs on the same vertex set, where $G_i\in\mathcal{G}$ and $\Delta(G_i)\leq d$ for all $i\in [m]$, admits a cooperative coloring.
\end{df}


\begin{df}{\upshape \cite{B}}
For a graph class $\mathcal{G}$, let $l_{\mathcal{G}}(d)$ be the minimum value of $l$ such that any family $G_1, G_2,\ldots, G_m$ of graphs, where $G_i\in\mathcal{G}$, $\Delta(G_i)\leq d$ for all $i\in [m]$, and each $v\in \bigcup^m_{i=1}V(G_i)$ belongs to at least $l$ graphs in this family, admits a cooperative list coloring.
\end{df}


Obviously, $m_{\mathcal{G}}(d)\leq l_{\mathcal{G}}(d)$ for any graph class $\mathcal{G}$.
A \emph{chordal graph} is defined as a graph in which every cycle of length greater than three contains at least one chord, which is an edge not part of the cycle. Aharoni et al.\ \cite{ABCHJ, ABZ, AHHS} explored the values of $m_{\mathcal{G}}(d)$ when $\mathcal{G}$ is represented by the class of chordal graphs, paths, trees, and bipartite graphs, respectively. In this paper, we specify the default base for logarithms as the natural number $e$ when the base is missing.

\begin{thm}{\upshape \cite{ABZ}}\label{chordal}
    Let $\mathcal{C}$ be the class of chordal graphs. Then $m_{\mathcal{C}}(d)=d+1$ for $d\geq 1$. 
\end{thm}

\begin{thm}{\upshape \cite{AHHS}}\label{path}
    Let $\mathcal{P}$ be the class of paths. Then $m_{\mathcal{P}}(2)=3$.
\end{thm}

\begin{thm}{\upshape \cite{ABCHJ}}\label{tree}
    Let $\mathcal{T}$ be the class of trees and $\mathcal{B}$ be the class of bipartite graphs. Then for $d\geq 2$, 
    $$\log_2\log_2d\leq m_{\mathcal{T}}(d)\leq (1+o(1))\log_{4/3}d,$$
    $$\log_2d\leq m_{\mathcal{B}}(d)\leq (1+o(1))\frac{2d}{\log d}.$$
\end{thm}
Let $\mathcal{F}$ be the graph class of forests.
Aharoni et al.\ \cite{ABCHJ} showed that $m_\mathcal{T}(d)=m_\mathcal{F}(d)$ for $d\geq 2$. It follows from $\mathcal{T}\subseteq \mathcal{F}$ that $m_\mathcal{T}(d)\leq m_\mathcal{F}(d)$. Conversely, consider $m:=m_\mathcal{T}(d)$ forests $F_1,F_2,\ldots,F_m$ with the maximum degree $d$. When $d\geq2$, we can augment each $F_i$ by adding edges to obtain a tree $F'_i$ that maintains the maximum degree, for every $i\in[m]$. By the definition of $m_\mathcal{T}(d)$, the graph family ${F'_1,F'_2,\ldots,F'_m}$ admits a cooperative coloring. This implies that the graph family ${F_1,F_2,\ldots,F_m}$ also admits a cooperative coloring, therefore $m_\mathcal{F}(d)\leq m$. 

Bradshaw and Masa\v{r}\'{i}k \cite{BM} investigated the upper bound on $m_{\mathcal{G}}(d)$ for the class $\mathcal{G}$ of degenerate graphs. Since every tree is a 1-degenerate graph, the following result generalizes Theorem \ref{tree} at the expense of a constant factor. 

\begin{thm}{\upshape \cite{BM}}\label{degenerate}
    Let $\mathcal{G}$ be the class of graphs and every graph in $\mathcal{G}$ is at most $k$-degenerate. Then $m_{\mathcal{G}}(d)\leq 13(1+k\log_2(kd))$.
\end{thm}

Bradshaw \cite{B} also studied the value $m_{\mathcal{S}}(d)$ for the class $\mathcal{S}$ of star forests, which improves the lower bound on $m_{\mathcal{T}}(d)$ in Theorem \ref{tree}.

\begin{thm}{\upshape \cite{B}}\label{star}
    Let $\mathcal{S}$ be the class of star forests. Then, $$ m_{\mathcal{S}}(d)\geq (1+o(1))\frac{\log d}{\log\log d}.$$
\end{thm}

In Section \ref{sec:2}, we study $m_{\mathcal{G}}(d)$ for small values of $d$ within specific graph classes $\mathcal{G}$. 
For the class of trees, denoted as $\mathcal{T}$, the exact value of $m_\mathcal{T}(d)$ remains undetermined, except for the cases where $d\leq 2$.
Specifically, when $d=3$, it can be deduced from Theorems \ref{chordal} and \ref{path} that $3\leq m_\mathcal{T}(3)\leq 4$ since $\mathcal{P}\subseteq \mathcal{T}\subseteq \mathcal{C}$. Although the proof of Theorem \ref{star} in \cite{B} gave a better construction to improve the lower bound of $m_\mathcal{T}(d)$, this construction method cannot prove that $m_\mathcal{T}(3)\geq 4$. We show that $m_\mathcal{T}(3)\geq 4$, leading to the conclusion that $m_\mathcal{T}(3)= 4$.
Furthermore, we prove that $m_\mathcal{W}(4)= 5$, in which $\mathcal{W}$ is the class of graphs whose components are wheels. 
In Section \ref{sec:3}, we study $m_{\mathcal{B}^*}(d)$ and $m_{\mathcal{L}}(d)$, where $\mathcal{B}^*$ represents the class of graphs whose components are balanced complete bipartite graphs, and $\mathcal{L}$ represents the class of graphs whose components are generalized theta graphs.
Specifically, we show that $m_{\mathcal{B}^*}(d)=O(\log_2 d)$ and $m_{\mathcal{L}}(d)=O\big(\frac{\log d}{\log\log d}\big)$.



\section{Trees and wheels}  \label{sec:2}

For the sake of clarity in our discussion, most of the proofs are presented in terms of adapted coloring. By showing that the edge-colored graph in Figure \ref{fig:td} does not admit an adapted coloring, we can establish the lower bound $m_\mathcal{T}(3)\geq 4$.

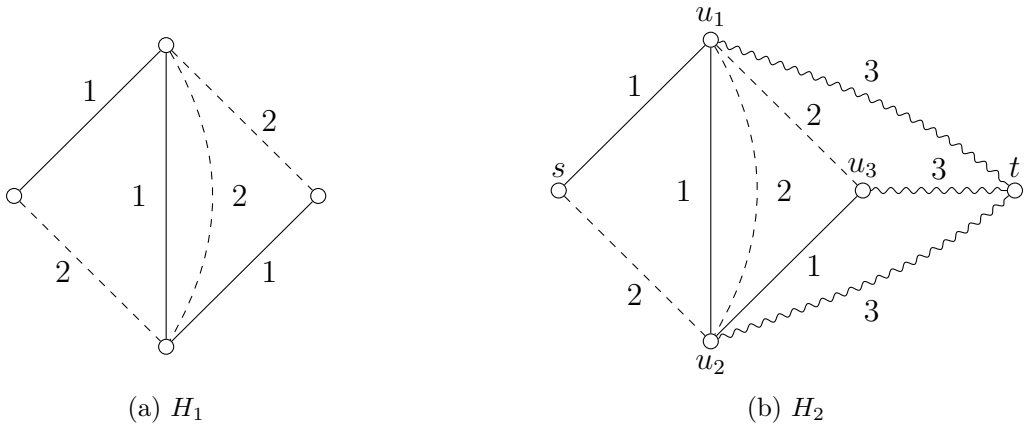
\begin{figure}[ht]
    \centering
    
    \subfloat[$H_1$]{
            \begin{tikzpicture}[inner sep=0.7mm]	
                \node[draw, circle](s0) at (0,0)[]{};
                \node[draw, circle](v1) at (2,2)[]{};
                \node[draw, circle](v2) at (2,-2)[label=below: {\rule{0pt}{5pt}}]{};
                \node[draw, circle](v3) at (4,0)[]{};
                 
                \draw[-] (s0) -- (v1) node[midway,above=5pt] {$1$};
                \draw[-] (v1) -- (v2) node[midway,left=5pt] {$1$};
                \draw[-] (v2) -- (v3) node[midway,right=5pt] {$1$};
                
                \draw[dashed] (s0) -- (v2) node[midway,left=5pt] {$2$};
                \draw[dashed] (v1) to[out=-60,in=60] node[midway,right=5pt] {$2$} (v2) ;
                \draw[dashed] (v1) -- (v3) node[midway,right=5pt] {$2$};
            \end{tikzpicture}
    }
    \hspace{6em}
    \subfloat[{$H_2$}]{
            \begin{tikzpicture}[inner sep=0.7mm]	
                \node[draw, circle](s) at (0,0)[label=above: $s$]{};
                \node[draw, circle](u1) at (2,2)[label=above: $u_1$]{};
                \node[draw, circle](u2) at (2,-2)[label=below: $u_2$]{};
                \node[draw, circle](u3) at (4,0)[label=above: $u_3$]{};
                \node[draw, circle](t) at (6,0)[label=above: $t$]{};
                 
                \draw[-] (s) -- (u1) node[midway,above=5pt] {$1$};
                \draw[-] (u1) -- (u2) node[midway,left=5pt] {$1$};
                \draw[-] (u2) -- (u3) node[midway,right=5pt] {$1$};
                
                \draw[dashed] (s) -- (u2) node[midway,below=5pt] {$2$};
                \draw[dashed] (u1) to[out=-60,in=60] node[midway,right=5pt] {$2$}  (u2);
                \draw[dashed] (u1) -- (u3) node[midway,right=5pt] {$2$};
                
                \draw[snake it] (u3) -- (t) node[midway,above=2pt] {$3$};
                \draw[snake it] (u2) to[out=20,in=-140] node[midway,below=5pt] {$3$} (t) ;
                \draw[snake it] (u1) to[out=-20,in=140] node[midway,above=5pt] {$3$} (t) ;
            \end{tikzpicture}
    }
    %
    %
    %
     	
    \caption{\label{fig:gc} In our construction illustrating that $m_\mathcal{T}(3)\geq 4$, these two graphs are elemental subgraphs. It is easy to check that $H_1$ does not admit an adapted coloring using color set $\{1,2\}$.}
\end{figure}

\begin{lem}\label{gadget}
   Let $\sigma$ be an adapted coloring of the edge-colored multigraph $H_2$ depicted in Figure \ref{fig:gc}(b). If $\sigma(s)\neq 3$, then $\sigma(t)\neq 3$.
\end{lem}

\begin{proof}
The subgraph $H_2[s,u_1,u_2,u_3]$ depicted in Figure \ref{fig:gc}(b) does not admit an adapted coloring when using the color set $\{1,2\}$. Consequently, color $3$ must be assigned to one of the vertices in $\{s,u_1,u_2,u_3\}$.
If $\sigma(s)\neq 3$, then at least one of $\{u_1,u_2,u_3\}$ must be colored with $3$. By the definition of adapted coloring, we can deduce that $\sigma(t)\neq 3$.
\end{proof}


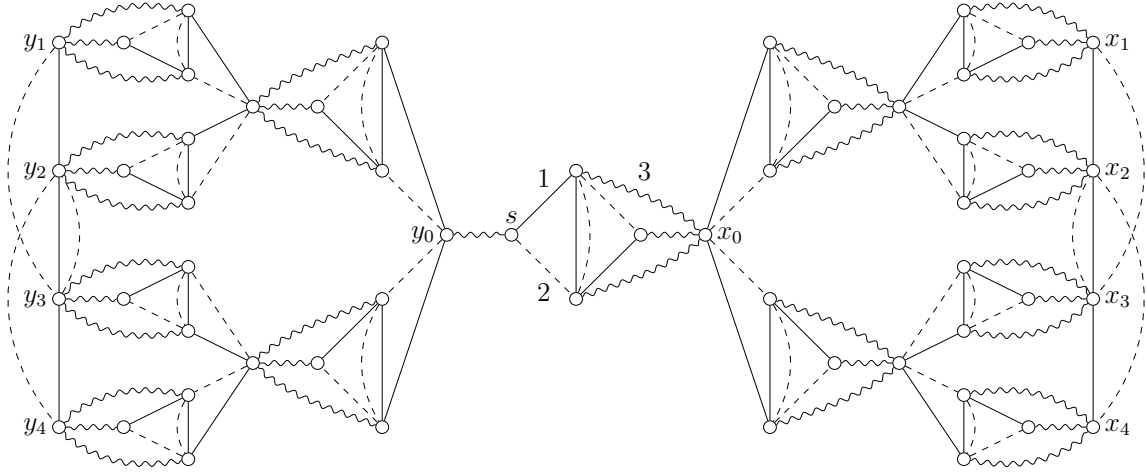
\begin{figure}[ht]\label{fig2}
    \centering
    
    \scalebox{0.85}{
    \begin{tikzpicture}[inner sep=0.7mm]	
        \node[draw, circle](t') at (-1,0)[label=left: $y_0$]{};
        \node[draw, circle](s) at (0,0)[label=above: $s$]{};
        \node[draw, circle](u1) at (1,1)[]{};
        \node[draw, circle](u2) at (1,-1)[]{};
        \node[draw, circle](u3) at (2,0)[]{};
        \node[draw, circle](t) at (3,0)[label=right: $x_0$]{};
         
        \draw[-] (s) --  node[midway,above=5pt] {$1$} (u1);
        \draw[-] (u1) -- (u2);
        \draw[-] (u2) -- (u3);
        
        \draw[dashed] (s) --  node[midway,below=5pt] {$2$} (u2);
        \draw[dashed] (u2) to[out=70,in=-70] (u1);
        \draw[dashed] (u1) -- (u3);
        
        \draw[snake it] (t') --  (s);
        \draw[snake it] (u3) -- (t);
        \draw[snake it] (u2) to[out=20,in=-140] (t);
        \draw[snake it] (u1) to[out=-20,in=140]  node[midway,above=5pt] {$3$} (t);
        
        \node[draw, circle](u4) at (4,1)[]{};
        \node[draw, circle](u5) at (4,3)[]{};
        \node[draw, circle](u6) at (5,2)[]{};
        \node[draw, circle](t1) at (6,2)[]{};
        
        \draw[-] (t) -- (u5);
        \draw[-] (u5) -- (u4);
        \draw[-] (u4) -- (u6);
        
        \draw[dashed] (t) -- (u4);
        \draw[dashed] (u4) to[out=60,in=-60] (u5);
        \draw[dashed] (u5) -- (u6);
        
        \draw[snake it] (u6) -- (t1);
        \draw[snake it] (u4) to[out=20,in=-140] (t1);
        \draw[snake it] (u5) to[out=-20,in=140] (t1);
        
        \node[draw, circle](u7) at (7,2.5)[]{};
        \node[draw, circle](u8) at (7,3.5)[]{};
        \node[draw, circle](u9) at (8,3)[]{};
        \node[draw, circle](t2) at (9,3)[label=right: $x_{1}$]{};
        
        \draw[-] (t1) -- (u8);
        \draw[-] (u8) -- (u7);
        \draw[-] (u7) -- (u9);
        
        \draw[dashed] (t1) -- (u7);
        \draw[dashed] (u7) to[out=60,in=-60] (u8);
        \draw[dashed] (u8) -- (u9);
        
        \draw[snake it] (u9) -- (t2);
        \draw[snake it] (u7) to[out=-10,in=-140] (t2);
        \draw[snake it] (u8) to[out=10,in=140] (t2);
        
        \node[draw, circle](u10) at (7,0.5)[]{};
        \node[draw, circle](u11) at (7,1.5)[]{};
        \node[draw, circle](u12) at (8,1)[]{};
        \node[draw, circle](t3) at (9,1)[label=right: $x_{2}$]{};
        
        \draw[-] (t1) -- (u11);
        \draw[-] (u11) -- (u10);
        \draw[-] (u10) -- (u12);
        
        \draw[dashed] (t1) -- (u10);
        \draw[dashed] (u10) to[out=60,in=-60] (u11);
        \draw[dashed] (u11) -- (u12);

        \draw[snake it] (u12) -- (t3);
        \draw[snake it] (u10) to[out=-10,in=-140] (t3);
        \draw[snake it] (u11) to[out=10,in=140] (t3);

        \node[draw, circle](u4r) at (4,-1)[]{};
        \node[draw, circle](u5r) at (4,-3)[]{};
        \node[draw, circle](u6r) at (5,-2)[]{};
        \node[draw, circle](t1r) at (6,-2)[]{};
        
        \draw[-] (t) -- (u5r);
        \draw[-] (u5r) -- (u4r);
        \draw[-] (u4r) -- (u6r);
        
        \draw[dashed] (t) -- (u4r);
        \draw[dashed] (u4r) to[out=-60,in=60] (u5r);
        \draw[dashed] (u5r) -- (u6r);

        \draw[snake it] (u6r) -- (t1r);
        \draw[snake it] (u4r) to[out=-20,in=140] (t1r);
        \draw[snake it] (u5r) to[out=20,in=-140] (t1r);
        
        \node[draw, circle](u7r) at (7,-2.5)[]{};
        \node[draw, circle](u8r) at (7,-3.5)[]{};
        \node[draw, circle](u9r) at (8,-3)[]{};
        \node[draw, circle](t2r) at (9,-3)[label=right: $x_{4}$]{};
        
        \draw[-] (t1r) -- (u8r);
        \draw[-] (u8r) -- (u7r);
        \draw[-] (u7r) -- (u9r);
        
        \draw[dashed] (t1r) -- (u7r);
        \draw[dashed] (u7r) to[out=-60,in=60] (u8r);
        \draw[dashed] (u8r) -- (u9r);
        
        \draw[snake it] (u9r) -- (t2r);
        \draw[snake it] (u7r) to[out=10,in=140] (t2r);
        \draw[snake it] (u8r) to[out=-10,in=-140] (t2r);
        
        \node[draw, circle](u10r) at (7,-0.5)[]{};
        \node[draw, circle](u11r) at (7,-1.5)[]{};
        \node[draw, circle](u12r) at (8,-1)[]{};
        \node[draw, circle](t3r) at (9,-1)[label=right: $x_{3}$]{};
        
        \draw[-] (t1r) -- (u11r);
        \draw[-] (u11r) -- (u10r);
        \draw[-] (u10r) -- (u12r);
        
        \draw[dashed] (t1r) -- (u10r);
        \draw[dashed] (u10r) to[out=-60,in=60] (u11r);
        \draw[dashed] (u11r) -- (u12r);
        
        \draw[snake it] (u12r) -- (t3r);
        \draw[snake it] (u10r) to[out=10,in=140] (t3r);
        \draw[snake it] (u11r) to[out=-10,in=-140] (t3r);
        \node[draw, circle](u4l) at (-2,1)[]{};
        \node[draw, circle](u5l) at (-2,3)[]{};
        \node[draw, circle](u6l) at (-3,2)[]{};
        \node[draw, circle](t1l) at (-4,2)[]{};
        
        \draw[-] (t') -- (u5l);
        \draw[-] (u5l) -- (u4l);
        \draw[-] (u4l) -- (u6l);
        
        \draw[dashed] (t') -- (u4l);
        \draw[dashed] (u4l) to[out=120,in=-120] (u5l);
        \draw[dashed] (u5l) -- (u6l);

        \draw[snake it] (u6l) -- (t1l);
        \draw[snake it] (u4l) to[out=160,in=-40] (t1l);
        \draw[snake it] (u5l) to[out=-160,in=40] (t1l);
        
        \node[draw, circle](u7l) at (-5,2.5)[]{};
        \node[draw, circle](u8l) at (-5,3.5)[]{};
        \node[draw, circle](u9l) at (-6,3)[]{};
        \node[draw, circle](t2l) at (-7,3)[label=left: $y_{1}$]{};
        
        \draw[-] (t1l) -- (u8l);
        \draw[-] (u8l) -- (u7l);
        \draw[-] (u7l) -- (u9l);
        
        \draw[dashed] (t1l) -- (u7l);
        \draw[dashed] (u7l) to[out=120,in=-120] (u8l);
        \draw[dashed] (u8l) -- (u9l);
        
        \draw[snake it] (u9l) -- (t2l);
        \draw[snake it] (u7l) to[out=-170,in=-40] (t2l);
        \draw[snake it] (u8l) to[out=170,in=40] (t2l);
        
        \node[draw, circle](u10l) at (-5,0.5)[]{};
        \node[draw, circle](u11l) at (-5,1.5)[]{};
        \node[draw, circle](u12l) at (-6,1)[]{};
        \node[draw, circle](t3l) at (-7,1)[label=left: $y_{2}$]{};
                 
        \draw[-] (t1l) -- (u11l);
        \draw[-] (u11l) -- (u10l);
        \draw[-] (u10l) -- (u12l);
        
        \draw[dashed] (t1l) -- (u10l);
        \draw[dashed] (u10l) to[out=120,in=-120] (u11l);
        \draw[dashed] (u11l) -- (u12l);
        
        \draw[snake it] (u12l) -- (t3l);
        \draw[snake it] (u10l) to[out=-170,in=-40] (t3l);
        \draw[snake it] (u11l) to[out=170,in=40] (t3l);
          
        \node[draw, circle](u4rl) at (-2,-1)[]{};
        \node[draw, circle](u5rl) at (-2,-3)[]{};
        \node[draw, circle](u6rl) at (-3,-2)[]{};
        \node[draw, circle](t1rl) at (-4,-2)[]{};
        
        \draw[-] (t') -- (u5rl);
        \draw[-] (u5rl) -- (u4rl);
        \draw[-] (u4rl) -- (u6rl);
        
        \draw[dashed] (t') -- (u4rl);
        \draw[dashed] (u4rl) to[out=-120,in=120] (u5rl);
        \draw[dashed] (u5rl) -- (u6rl);
        
        \draw[snake it] (u6rl) -- (t1rl);
        \draw[snake it] (u4rl) to[out=-160,in=40] (t1rl);
        \draw[snake it] (u5rl) to[out=160,in=-40] (t1rl);
        
        \node[draw, circle](u7rl) at (-5,-2.5)[]{};
        \node[draw, circle](u8rl) at (-5,-3.5)[]{};
        \node[draw, circle](u9rl) at (-6,-3)[]{};
        \node[draw, circle](t2rl) at (-7,-3)[label=left: $y_{4}$]{};
        
        \draw[-] (t1rl) -- (u8rl);
        \draw[-] (u8rl) -- (u7rl);
        \draw[-] (u7rl) -- (u9rl);
        
        \draw[dashed] (t1rl) -- (u7rl);
        \draw[dashed] (u7rl) to[out=-120,in=120] (u8rl);
        \draw[dashed] (u8rl) -- (u9rl);

        \draw[snake it] (u9rl) -- (t2rl);
        \draw[snake it] (u7rl) to[out=170,in=40] (t2rl);
        \draw[snake it] (u8rl) to[out=-170,in=-40] (t2rl);
        
        \node[draw, circle](u10rl) at (-5,-0.5)[]{};
        \node[draw, circle](u11rl) at (-5,-1.5)[]{};
        \node[draw, circle](u12rl) at (-6,-1)[]{};
        \node[draw, circle](t3rl) at (-7,-1)[label=left: $y_{3}$]{};
                 
        \draw[-] (t1rl) -- (u11rl);
        \draw[-] (u11rl) -- (u10rl);
        \draw[-] (u10rl) -- (u12rl);
        
        \draw[dashed] (t1rl) -- (u10rl);
        \draw[dashed] (u10rl) to[out=-120,in=120] (u11rl);
        \draw[dashed] (u11rl) -- (u12rl);

        \draw[snake it] (u12rl) -- (t3rl);
        \draw[snake it] (u10rl) to[out=170,in=40] (t3rl);
        \draw[snake it] (u11rl) to[out=-170,in=-40] (t3rl);
          
        \draw[-] (t2) -- (t3);
        \draw[-] (t3) -- (t3r);
        \draw[-] (t3r) -- (t2r);
        
        \draw[dashed] (t2)  to[out=-50,in=50] (t3r);
        \draw[dashed] (t3r) to[out=120,in=-120] (t3);
        \draw[dashed] (t3)  to[out=-50,in=50] (t2r);
        
        \draw[-] (t2l) -- (t3l);
        \draw[-] (t3l) -- (t3rl);
        \draw[-] (t3rl) -- (t2rl);
        
        \draw[dashed] (t2l)  to[out=-130,in=130] (t3rl);
        \draw[dashed] (t3rl) to[out=60,in=-60] (t3l);
        \draw[dashed] (t3l)  to[out=-130,in=130] (t2rl);
    \end{tikzpicture}
     }
     \caption{\label{fig:td}This is a graph that has been edge-colored with three colors and does not admit an adapted coloring.}
\end{figure}

\begin{thm}\label{T3} 
    Let $\mathcal{T}$ be the class of trees. Then $m_\mathcal{T}(3)\geq 4$.
\end{thm}

\begin{proof}
    Let $(G,\phi)$ be the edge-colored graph using colors from $\{1,2,3\}$ in Figure \ref{fig:td}. Observe that each monochromatic induced subgraph of $(G,\phi)$ is a forest of maximum degree 3. Suppose there exists an adapted coloring $\sigma$ of $(G,\phi)$. 
    
    If $\sigma(s)\neq 3$, then $\sigma(x_i)\neq3$ for $1\leq i\leq 4$ by Lemma \ref{gadget}. Since $G[x_1,x_2,x_3,x_4]$ is isomorphic to $H_1$ in Figure \ref{fig:gc}(a), it is evident that $G[x_1,x_2,x_3,x_4]$ does not admit an adapted coloring using colors from $\{1,2,3\}\backslash\{3\}$, a contradiction. If $\sigma(s)=3$, then $\sigma(y_0)\neq3$. Therefore, $\sigma(y_i)\neq3$ for $1\leq i\leq 4$, by Lemma \ref{gadget}. Similarly, $G[y_1,y_2,y_3,y_4]$ does not admit an adapted coloring using colors from $\{1,2,3\}\backslash\{3\}$, a contradiction. 
\end{proof}

It follows from Theorems \ref{chordal} and \ref{path} that $3\leq m_\mathcal{T}(3)\leq 4$ since $\mathcal{P}\subseteq\mathcal{T}\subseteq\mathcal{C}$. Together with Theorem \ref{T3}, we have the following conclusion.

\begin{cor}
    $m_\mathcal{T}(3)= 4$.
\end{cor}

A \emph{wheel} $W_n$ is the join of a cycle $C_n$ and a single vertex, i.e., $W_n\cong C_n \vee K_1$, where $n\geq 3$. It is worth mentioning that the construction in the proof of the lower bound of Theorem \ref{w4} is similar to the construction presented in Theorem 1 of \cite{B}.

\begin{thm}\label{w4} 
    Let $\mathcal{W}$ be the class of graphs whose components are wheels.
    Then, $$m_\mathcal{W}(4)=5.$$
\end{thm}

\begin{proof}
First, we show that $m_\mathcal{W}(4)\leq 5$. Given a graph family $G_1,G_2,\ldots,G_5$ of $\mathcal{W}$ on the same vertex set, where $\Delta(G_i)\leq 4$ for $i\in [5]$, it can be checked that each $G_i$ is either a $C_3\vee K_1 (\cong K_4)$ or a $C_4\vee K_1(\subseteq K_5-e)$. Both $K_4$ and $K_5-e$ are chordal graphs. Thus each $G_i$ is a subgraph of a chordal graph $G^*_i$ with the maximum degree at most 4. It follows from Theorem \ref{chordal} that the graph family $G^*_1, G^*_2,\ldots, G^*_5$ admits a cooperative coloring. Consequently, the graph family $G_1, G_2,\ldots, G_5$ admits a cooperative coloring.

Next, we construct an edge-colored graph $G$ with a color set $\{1,2,3,4\}$ such that its maximal monochromatic subgraphs are graphs in $\mathcal{W}$ with the maximum degree at most $4$. Afterward, we show that $G$ does not admit an adapted coloring. By the relation of adapted coloring and cooperative coloring, we get $m_\mathcal{W}(4)\geq 5$.

    \begin{figure}[ht]
        \centering
        \scalebox{0.9}{
        \subfloat[{$W_3$}]{
                \begin{tikzpicture}[inner sep=0.7mm]	
                    \node[draw, circle](s) at (0,0)[]{};
                    \node[draw, circle](u1) at (0,2)[]{};
                    \node[draw, circle](u2) at (2,2)[]{};
                    \node[draw, circle](u3) at (2,0)[]{};
                     
                    \draw[loosely dotted,line width = 1.5pt] (s) -- (u1);
                    \draw[loosely dotted,line width = 1.5pt] (s) -- (u2);
                    \draw[loosely dotted,line width = 1.5pt] (s) -- (u3);
                    \draw[loosely dotted,line width = 1.5pt] (u1) -- (u3);
                    \draw[loosely dotted,line width = 1.5pt] (u2) -- (u3);
    
                    \draw[-] (u1) -- node[midway,below=3pt] {$1$} (u2);
                    \draw[dashed] (u1) to[out=30,in=150] node[midway,fill=white] {$2$} (u2);
                    \draw[snake it] (u1) to[out=60,in=120] node[midway,above=5pt] {$3$} (u2);
                \end{tikzpicture}
        }
        \hspace{3em}
        \subfloat[{$H_1$}]{
                \begin{tikzpicture}[inner sep=0.7mm]	
                    \node[draw, circle](s) at (0,0)[]{};
                    \node[draw, circle](u1) at (0,2)[]{};
                    \node[draw, circle](u2) at (2,2)[]{};
                    \node[draw, circle](u3) at (2,0)[]{};
                     
                    \draw[loosely dotted,line width = 1.5pt] (s) -- (u1);
                    \draw[loosely dotted,line width = 1.5pt] (s) -- (u2);
                    \draw[loosely dotted,line width = 1.5pt] (s) -- (u3);
                    \draw[loosely dotted,line width = 1.5pt] (u1) -- (u3);
                    \draw[loosely dotted,line width = 1.5pt] (u2) -- (u3);
    
                    \draw[-] (u1) -- node[midway,below=3pt] {$2$} (u2);
                    \draw[dashed] (u1) to[out=30,in=150] node[midway,fill=white] {$3$} (u2);
                    \draw[snake it] (u1) to[out=60,in=120] node[midway,above=5pt] {$4$} (u2);
                \end{tikzpicture}
        }
        \hspace{0.5em}
        \subfloat[{$H_2$}]{
                \begin{tikzpicture}[inner sep=0.7mm]	
                    \node[draw, circle](s) at (0,0)[]{};
                    \node[draw, circle](u1) at (0,2)[]{};
                    \node[draw, circle](u2) at (2,2)[]{};
                    \node[draw, circle](u3) at (2,0)[]{};
                     
                    \draw[loosely dotted,line width = 1.5pt] (s) -- (u1);
                    \draw[loosely dotted,line width = 1.5pt] (s) -- (u2);
                    \draw[loosely dotted,line width = 1.5pt] (s) -- (u3);
                    \draw[loosely dotted,line width = 1.5pt] (u1) -- (u3);
                    \draw[loosely dotted,line width = 1.5pt] (u2) -- (u3);
    
                    \draw[-] (u1) -- node[midway,below=3pt] {$1$} (u2);
                    \draw[dashed] (u1) to[out=30,in=150] node[midway,fill=white] {$3$} (u2);
                    \draw[snake it] (u1) to[out=60,in=120] node[midway,above=5pt] {$4$} (u2);
                \end{tikzpicture}
        }
        \hspace{0.5em}
        \subfloat[{$H_3$}]{
                \begin{tikzpicture}[inner sep=0.7mm]	
                    \node[draw, circle](s) at (0,0)[]{};
                    \node[draw, circle](u1) at (0,2)[]{};
                    \node[draw, circle](u2) at (2,2)[]{};
                    \node[draw, circle](u3) at (2,0)[]{};
                     
                    \draw[loosely dotted,line width = 1.5pt] (s) -- (u1);
                    \draw[loosely dotted,line width = 1.5pt] (s) -- (u2);
                    \draw[loosely dotted,line width = 1.5pt] (s) -- (u3);
                    \draw[loosely dotted,line width = 1.5pt] (u1) -- (u3);
                    \draw[loosely dotted,line width = 1.5pt] (u2) -- (u3);
    
                    \draw[-] (u1) -- node[midway,below=3pt] {$1$} (u2);
                    \draw[dashed] (u1) to[out=30,in=150] node[midway,fill=white] {$2$} (u2);
                    \draw[snake it] (u1) to[out=60,in=120] node[midway,above=5pt] {$4$} (u2);
                \end{tikzpicture}
        }
        \hspace{0.5em}
        \subfloat[{$H_4$}]{
                \begin{tikzpicture}[inner sep=0.7mm]	
                    \node[draw, circle](s) at (0,0)[]{};
                    \node[draw, circle](u1) at (0,2)[]{};
                    \node[draw, circle](u2) at (2,2)[]{};
                    \node[draw, circle](u3) at (2,0)[]{};
                     
                    \draw[loosely dotted,line width = 1.5pt] (s) -- (u1);
                    \draw[loosely dotted,line width = 1.5pt] (s) -- (u2);
                    \draw[loosely dotted,line width = 1.5pt] (s) -- (u3);
                    \draw[loosely dotted,line width = 1.5pt] (u1) -- (u3);
                    \draw[loosely dotted,line width = 1.5pt] (u2) -- (u3);
    
                    \draw[-] (u1) -- node[midway,below=3pt] {$1$} (u2);
                    \draw[dashed] (u1) to[out=30,in=150] node[midway,fill=white] {$2$} (u2);
                    \draw[snake it] (u1) to[out=60,in=120] node[midway,above=5pt] {$3$} (u2);
                \end{tikzpicture}
        }
        }
        \caption{\label{fig:w3Hi}In the graph depicted, each dotted line represents three edges that are colored differently. For instance, in $W_3$, every pair of vertices is connected by three edges, each assigned a distinct edge color from $\{1,2,3\}$. The graphs labeled as $H_1$, $H_2$, $H_3$, and $H_4$ are copies of $W_3$ with edge colorings that have been shifted.}
    \end{figure}
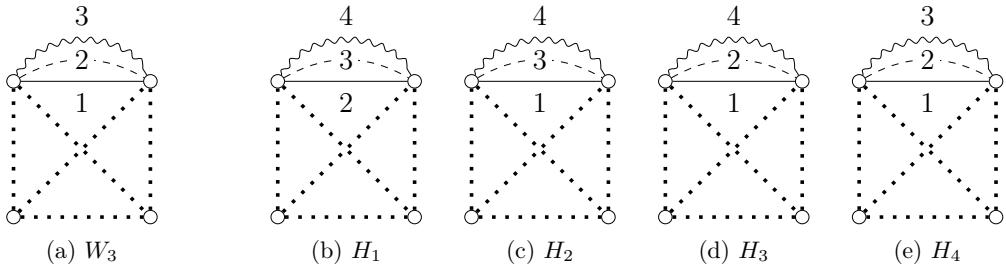
    
    We denote the edge-colored multigraph depicted in Figure \ref{fig:w3Hi}(a) by $W_3$, and denote its edge-coloring be $\phi$. It can be verified that $(W_3,\phi)$ does not admit an adapted coloring using the color set $\{1,2,3\}$. For $1\leq i\leq 4$, the shift function $\psi_i:\{1,2,3\}\to\{1,2,3,4\}$ is defined as:
    $$
    \psi_i(x)=
    \begin{cases}
    x, &1\le x\le i-1,\\
    x+1, &i\le x \le 3.\\
    \end{cases}
    $$
    By applying these shift functions, we can generate four disjoint copies of $W_3$ denoted as $H_1$, $H_2$, $H_3$, and $H_4$ (as shown in Figure \ref{fig:w3Hi}). Each $H_i$ is edge-colored using the function $\psi_i\circ\phi$. It is easy to see that $(H_i,\psi_i\circ\phi)$ is isomorphic to $(W_3,\phi)$ as an edge-colored graph. Therefore, $(H_i,\psi_i\circ\phi)$ does not admit an adapted coloring using the color set $\{1,2,3,4\}\setminus\{i\}$.

    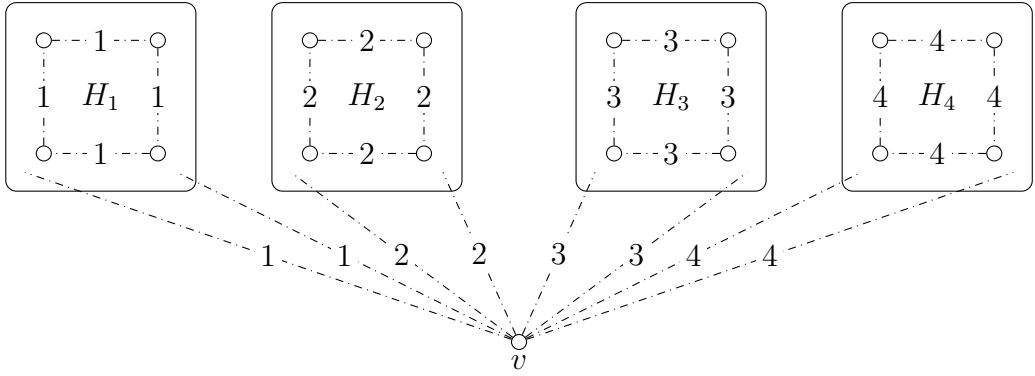
\begin{figure}[ht]
        \centering
                \begin{tikzpicture}[inner sep=0.7mm]	
                    \node[draw, circle](u0) at (0,0)[]{};
                    \node[draw, circle](u1) at (0,1.5)[]{};
                    \node[draw, circle](u2) at (1.5,1.5)[]{};
                    \node[draw, circle](u3) at (1.5,0)[]{};
                     
                    \draw[dash dot] (u0) -- node[midway,fill=white] {$1$} (u1);
                    \draw[dash dot] (u1) -- node[midway,fill=white] {$1$} (u2);
                    \draw[dash dot] (u2) -- node[midway,fill=white] {$1$} (u3);
                    \draw[dash dot] (u3) -- node[midway,fill=white] {$1$} (u0);
                    \draw[rounded corners] (-0.5, -0.5) rectangle (2, 2) {};
                    \node at (0.75, 0.75)   (H1) {$H_1$};
                    \node[draw, circle](v0) at (3.5,0)[]{};
                    \node[draw, circle](v1) at (3.5,1.5)[]{};
                    \node[draw, circle](v2) at (5,1.5)[]{};
                    \node[draw, circle](v3) at (5,0)[]{};
                     
                    \draw[dash dot] (v0) -- node[midway,fill=white] {$2$} (v1);
                    \draw[dash dot] (v1) -- node[midway,fill=white] {$2$} (v2);
                    \draw[dash dot] (v2) -- node[midway,fill=white] {$2$} (v3);
                    \draw[dash dot] (v3) -- node[midway,fill=white] {$2$} (v0);
                    \draw[rounded corners] (3, -0.5) rectangle (5.5, 2) {};
                    \node at (4.25, 0.75)   (H2) {$H_2$};
                    \node[draw, circle](x0) at (7.5,0)[]{};
                    \node[draw, circle](x1) at (7.5,1.5)[]{};
                    \node[draw, circle](x2) at (9,1.5)[]{};
                    \node[draw, circle](x3) at (9,0)[]{};
                     
                    \draw[dash dot] (x0) -- node[midway,fill=white] {$3$} (x1);
                    \draw[dash dot] (x1) -- node[midway,fill=white] {$3$} (x2);
                    \draw[dash dot] (x2) -- node[midway,fill=white] {$3$} (x3);
                    \draw[dash dot] (x3) -- node[midway,fill=white] {$3$} (x0);
                    \draw[rounded corners] (7, -0.5) rectangle (9.5, 2) {};
                    \node at (8.25, 0.75)   (H3) {$H_3$};
                    \node[draw, circle](y0) at (11,0)[]{};
                    \node[draw, circle](y1) at (11,1.5)[]{};
                    \node[draw, circle](y2) at (12.5,1.5)[]{};
                    \node[draw, circle](y3) at (12.5,0)[]{};
                     
                    \draw[dash dot] (y0) -- node[midway,fill=white] {$4$} (y1);
                    \draw[dash dot] (y1) -- node[midway,fill=white] {$4$} (y2);
                    \draw[dash dot] (y2) -- node[midway,fill=white] {$4$} (y3);
                    \draw[dash dot] (y3) -- node[midway,fill=white] {$4$} (y0);
                    \draw[rounded corners] (10.5, -0.5) rectangle (13, 2) {};
                    \node at (11.75, 0.75)   (H4) {$H_4$};
                    \node[draw, circle](v) at (6.25,-2.5)[label=below: $v$]{};
    
                    \draw[dash dot] (v) -- node[midway,fill=white] {$1$} (-0.25,-0.25);
                    \draw[dash dot] (v) -- node[midway,fill=white] {$1$} (1.75,-0.25);
                    \draw[dash dot] (v) -- node[midway,fill=white] {$2$} (3.25,-0.25);
                    \draw[dash dot] (v) -- node[midway,fill=white] {$2$} (5.25,-0.25);
                    \draw[dash dot] (v) -- node[midway,fill=white] {$3$} (7.25,-0.25);
                    \draw[dash dot] (v) -- node[midway,fill=white] {$3$} (9.25,-0.25);
                    \draw[dash dot] (v) -- node[midway,fill=white] {$4$} (10.75,-0.25);
                    \draw[dash dot] (v) -- node[midway,fill=white] {$4$} (12.75,-0.25);
                \end{tikzpicture}
        \caption{\label{fig:W4}An edge-colored multigraph using four colors, where vertex $v$ is connected to each vertex in $H_i$ by an edge colored $i$ for each $1\leq i\leq 4$. This graph does not admit an adapted coloring.}
    \end{figure}

The edge-colored multigraph $G$ (see Figure \ref{fig:W4}) is obtained by the following construction. 
First, add a 4-cycle $C_i$ in $H_i$ for each $i\in [4]$. At the same time, add a new vertex $v$ and join $v$ with each vertex of $\cup_{i=1}^{4}H_i$. Next, assign color $i$ to edges between to $v$ and $H_i$ for $i\in[4]$.
In this construction, each component of the monochromatic subgraphs of $G$ is either a $W_3$ or a $W_4$.
In any adapted coloring of $G$ using color set $\{1,2,3,4\}$, at least one vertex of $H_i$ must be colored with $i$. It implies that no color in $\{1,2,3,4\}$ is available at $v$. Therefore, $G$ does not admit an adapted coloring using color set $\{1,2,3,4\}$.
\end{proof}

\section{Bipartite graphs and generalized theta
graphs} \label{sec:3}

Aharoni et al. \cite{ABCHJ} investigated cooperative coloring of the class of bipartite graphs and showed that $m_\mathcal{B}(d)\geq \log_2 d$ (as stated in Theorem \ref{tree}). A complete bipartite graph $K_{m,n}$ is called \emph{balanced} if $m=n$.
Let $\mathcal{B}^*$ be the class of graphs in which each component is a balanced complete bipartite graph, and we can get $m_{\mathcal{B}^*}(d)\geq \log_2 d$ by the constructive proof of lower bound on bipartite graphs in \cite{ABCHJ}.
Now we show that $\log_2 d$ is asymptotically the best possible for $m_{\mathcal{B}^*}(d)$.  The following famous Lov\'{a}sz Local Lemma first appears in a weaker form in \cite{EL} and can be found in many textbooks, including for example \cite[Chapter 4]{MB}.

\begin{thm}{\upshape \cite{MB}}\label{LocalLem}
    Let $\mathcal{B}$ be a set of bad events. Suppose that each event $B\in \mathcal{B}$ occurs with probability at most $p$, and suppose further that each event $B\in \mathcal{B}$ is independent with all but at most $d$ other events $B'\in \mathcal{B}$. If $epd\leq 1$,
    then with positive probability, no bad event in $\mathcal{B}$ occurs.
\end{thm}

The proof of the following theorem could be proceeded using the idea of choosing independent sets in \cite{LZ}. Nevertheless, by noting that each graph in $\mathcal{B}^*$ is a bipartite graph and each part of its components is an independent set, we present here a bit different but more straightforward and concise proof.

\begin{thm}
    For $d\geq 2$, there holds $m_{\mathcal{B}^*}(d)\leq (1+o(1))\log_2 d$. 
\end{thm}

\begin{proof}
    Consider a family ${G_1, G_2, \ldots, G_m}$ of graphs on the same vertex set $V$, where the components of each $G_i$ are balanced complete bipartite graphs and $\Delta(G_i)\leq d$ for $1\leq i\leq m$. 
    
    
      Let $X_i$ be a set of vertices by choosing one part from each complete bipartite component uniformly at random in $G_i$ for each $1 \leq i \leq m$.
     It is evident that $X_i$ is an independent set in $G_i$. 
    We now show that with a positive probability, $X_1, X_2, \ldots, X_m$ collectively constitute a cooperative coloring of the graph family ${G_1, G_2, \ldots, G_m}$. For each vertex $v \in V$, let $B_v$ be the event that $v\notin \cup^m_{i=1}{X_i}$. 
    Then we get $\Pr(v\notin \cup^m_{i=1}{X_i})=(1/2)^m$ since 
    $\Pr(v\in X_i)=\frac{1}{2}.$ 
  It follows from the degree of $v$ is at most $d$ that $B_v$ is independent with all but at most $2md$ other events. By applying the Lov\'{a}sz Local Lemma (Theorem \ref{LocalLem}), if 
    \begin{equation}
        e\times \bigg(\frac{1}{2}\bigg)^m \times 2md\leq 1,
    \end{equation}
    then no $B_v$ occurs with positive probability, meaning that the sets $X_1, X_2,\ldots, X_m$ form a cooperative coloring.
    The inequality holds when $m\geq (1+o(1))\log_2 d$.
\end{proof}	

The \emph{generalized theta graph} $\theta_{s_1,\ldots,s_k}$ consists of a pair of vertices joined by $k$ internally disjoint paths of lengths $s_1,\ldots,s_k$, where each $s_i\geq 1$.
Let $\mathcal{L}$ be the class of graphs whose components are generalized theta graphs. 
To study the value of $m_{\mathcal{L}}(d)$, we first introduce some definitions and a lemma. Given a rooted tree $T$ with a root $r$, the {\em height} of a vertex $v$ in $T$ is the distance from $v$ to $r$, and the height of $T$ is the maximum height achieved over all vertices $v\in V(T)$. Given integers $q\geq 1$ and
$h\geq 1$, a {\em $q$-ary tree of height $h$} is a rooted tree in which every vertex of height at most $h-1$ has exactly $q$ children. Given an integer $k\geq 1$, we write $\log^{(k)}d = \underbrace{\log \log\ldots \log }_{k \ \text{times}} d$.



\begin{lem}{\upshape \cite{B}\label{lem2}}
Let $q\geq 2$ and $h\geq1$ be fixed integers. If $\mathcal{H}$ is a family of graphs with no
$q$-ary tree of height $h$ as a subgraph, then
$$l_{\mathcal{H}}(d) \leq (1 + o_{q,h}(1))\frac{\log d}{\log^{(h)}d}+ O_q(1).$$
\end{lem}

\begin{thm}{\upshape\label{Bi}}
   For $d \geq 2$, $m_{\mathcal{L}}(d)=(1+o(1))\frac{\log d}{\log \log d}$.
\end{thm}

\begin{proof}
Since no generalized theta graph contains a $3$-ary tree of height $2$ as a subgraph, we have $m_{\mathcal{L}}(d)\leq l_{\mathcal{L}}(d)\leq (1+o(1))\frac{\log d}{\log\log d}$ by Lemma \ref{lem2}. In the following, we discuss the lower bound of $m_{\mathcal{L}}(d)$.

For each $t \geq 1$, we construct a graph $G_t$ whose edges are colored with $\{1, 2, \ldots, t\}$ using some function $\phi_t$. In this graph, for each $i\in [t]$, every component of the monochromatic subgraph induced by the edges of color $i$ is a $K_{2,s}$ with $s\geq 2$, which is a generalized theta graph. We then translate the edge-colored graph $(G_t, \phi_t)$ into a graph family that demonstrates our lower bound. It is worth mentioning that the construction here uses ideas from \cite{B}.

\begin{figure}[ht]
    \centering
            \begin{tikzpicture}[inner sep=0.7mm]	
                \node[draw, circle](u0) at (0,0)[]{};
                \node[draw, circle](u1) at (0,1.5)[]{};
                \node[draw, circle](u2) at (1.5,1.5)[]{};
                \node[draw, circle](u3) at (1.5,0)[]{};
                 
                \draw (u0) -- node[midway,left=5pt] {$1$} (u1);
                \draw (u1) -- node[midway,above=5pt] {$1$} (u2);
                \draw (u2) -- node[midway,right=5pt] {$1$} (u3);
                \draw (u3) -- node[midway,below=5pt] {$1$} (u0);
                \node at (0.75, -2)   (G1) {($G_1$,$\phi_1$)};
                \node[draw, circle](x0) at (9,0)[]{};
                \node[draw, circle](x1) at (9,1.5)[]{};
                \node[draw, circle](x2) at (10.5,1.5)[]{};
                \node[draw, circle](x3) at (10.5,0)[]{};
                 
                \draw (x0) -- (x1);
                \draw (x1) -- (x2);
                \draw (x2) -- node[midway,right=5pt] {$1$} (x3);
                \draw (x3) -- (x0);
                \node[draw, circle](y0) at (5.5,0)[]{};
                \node[draw, circle](y1) at (5.5,1.5)[]{};
                \node[draw, circle](y2) at (7,1.5)[]{};
                \node[draw, circle](y3) at (7,0)[]{};
                 
                \draw[dashed] (y0) -- node[midway,left=5pt] {$2$} (y1);
                \draw[dashed] (y1) --  (y2);
                \draw[dashed] (y2) -- (y3);
                \draw[dashed] (y3) -- (y0);
                \node[draw, circle](u) at (8,3)[]{};
                \node[draw, circle](v) at (8,-1.5)[]{};
                \node at (8, -2)   (G1) {$(G_2,\phi_2)$};

                \draw[dashed] (u) -- (x0);
                \draw[dashed] (u) -- (x1);
                \draw[dashed] (u) --  node[midway,above=5pt] {$2$} (x2);
                \draw[dashed] (u) -- (x3);

                \draw (u) -- (y0);
                \draw (u) --  node[midway,above=5pt] {$1$} (y1);
                \draw (u) -- (y2);
                \draw (u) -- (y3);
                
                \draw[dashed] (v) -- (x0);
                \draw[dashed] (v) -- (x1);
                \draw[dashed] (v) -- (x2);
                \draw[dashed] (v) -- node[midway,below=5pt] {$2$} (x3);

                \draw (v) --  node[midway,below=5pt] {$1$} (y0);
                \draw (v) -- (y1);
                \draw (v) -- (y2);
                \draw (v) -- (y3);
            \end{tikzpicture}
    \caption{\label{G1G2}The first two graphs of the recursively constructed graphs $(G_t,\phi_t)$.}
\end{figure}
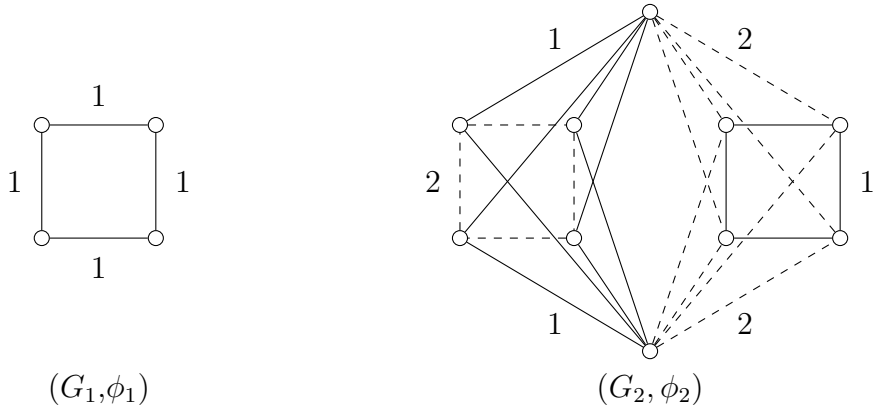

We construct the edge-colored graphs $(G_t,\phi_t)$ recursively, as illustrated in Figure \ref{G1G2} for $(G_1,\phi_1)$ and $(G_2,\phi_2)$. Let $(G_1,\phi_1)$ be a $K_{2,2}$ whose edge is colored with the color 1. Assume that we have constructed $G_t$ along with an edge-coloring $\phi_t:E(G_t)\rightarrow \{1,2,\ldots,t\}$, and assume that $(G_t,\phi_t)$ does not admit an adapted coloring with the color set $\{1,\ldots,t\}$. For $1\leq i\leq t+1$, we define a shift function $\psi_{i}: \{1,2,\ldots,t\}\rightarrow\{1,2,\ldots,t+1\}$ such that
$$
\psi_i(x)=
    \begin{cases}
        x, &1\le x\le i-1,\\
        x+1, &i\le x \le t.\\
    \end{cases}
$$
We begin to construct $(G_{t+1},\phi_{t+1})$ as follows: first create $t+1$ disjoint copies, $H_1, H_2, \ldots, H_{t+1}$, of $G_t$, where each $H_i$ is edge-colored using the function $\psi_i \circ \phi_t$. Note that $(H_i,\psi_i\circ\phi_t)$ is isomorphic to $(G_t,\phi_t)$ as an edge-colored graph. Therefore, $(H_i,\psi_i\circ\phi_t)$ does not admit an adapted coloring with colors from $\{1,2,\ldots,i-1,i+1,\ldots,t+1\}$. 
Next, we add two new vertices, $u$ and $v$, and
add an edge of color $i$ joining each vertex of $\{u,v\}$ and each vertex of $H_i$ for $1\leq i\leq t+1$. We denote the resulting graph as $G_{t+1}$ and its edge coloring as $\phi_{t+1}$. With this construction, every component of the monochromatic subgraph induced by the edges of each color in $(G_{t+1},\phi_{t+1})$ is a $K_{2,s}$ with $s\geq 2$. If $(G_{t+1},\phi_{t+1})$ admits an adapted coloring with the color set $[t+1]$, then $H_i$ must contain a vertex colored $i$, and no color from $[t+1]$ is available at either $u$ or $v$. Consequently, $(G_{t+1},\phi_{t+1})$ does not admit an adapted coloring.

Now, we determine the maximum degree of each monochromatic subgraph in $G_t$. We define $V_t$ as the number of vertices in $G_t$ and $\Delta_t$ as the maximum degree in all monochromatic subgraphs of $G_t$. For $G_1$, we have $V_1 = 4$ and $\Delta_1 = 2$. Moving on to the recursive case for $t\geq 2$:
$$V_t=tV_{t-1}+2, \quad \Delta_t=V_{t-1}.$$
Solving the recurrence, we have that 
\begin{eqnarray*}
    V_t &=& V_1\cdot t!+2\cdot\frac{t!}{2!}+2\cdot\frac{t!}{3!}+\cdots+2\cdot\frac{t!}{(t-2)!}+2\cdot\frac{t!}{(t-1)!}+2\\
    &=& 2\cdot t!\cdot\bigg(1+1+\frac{1}{2!}+\frac{1}{3!}+\cdots+\frac{1}{(t-2)!}+\frac{1}{(t-1)!}+\frac{1}{t!}\bigg)\\
    &=& \big(2e+o(1)\big)\cdot t!.
\end{eqnarray*}
Moreover, $\Delta_t=V_{t-1}=\big(2e+o(1)\big)(t-1)!$.

Now, consider a value $d$, and choose $t$ so that  $\Delta_t \leq d < \Delta_{t+1}$. We construct $(G_t, \phi_t)$ as above, and let $G^i_t$ be the monochromatic subgraph of $(G_t, \phi_t)$ with color $i$ for $1 \leq i \leq t$. Note that each component of $G^i_t$ can be regarded as a generalized theta graph with the maximum degree at most $d$. Since $(G_t, \phi_t)$ does not admit an adapted coloring, the graph family ${G^1_t, G^2_t, \ldots, G^t_t}$ does not admit a cooperative coloring.

To complete the proof, it suffices to show that $d \leq (2e+o(1))t!$ implies $t \geq (1+o(1))\frac{\log d}{\log \log d}$.
Let 
$$d=(2e+o(1))t!=(2e+o(1))\sqrt{2\pi t}\bigg(\frac{t}{e}\bigg)^t,$$
and we have 
$$
    \log d = \log (2e+o(1)) +\frac{1}{2}\log 2\pi t + t\log t-t = t(\log t -1 + o(1)),
$$
$$
    \log\log d = \log t +\log\big(\log t -1 + o(1)\big).
$$
There follows
$$
    t\cdot\frac{\log\log d}{\log d} = \frac{\log t +\log\big(\log t -1 + o(1)\big)}{\log t -1 + o(1)} = 1 + o(1),
$$
as required.
\end{proof}



\section{Remark}
Let $\mathcal{W}^*$ be the class of graphs whose components are wheel graphs ($C_m\vee K_1$) or fan graphs ($P_n\vee K_1$), where $m\geq3$ and $n\geq 1$. A \emph{caterpillar} is defined as a tree in which, upon removing all the pendant vertices, it results in a path. 
A \emph{ring star} is a graph
that can be decomposed into a cycle (or ring) and a set of vertices each of them not belonging to the cycle but connected to it through a single edge.
Let $\mathcal{M}$ be the class of caterpillar graphs, and $\mathcal{R}$ be the class of ring star graphs. 
Observe that each graph in $\mathcal{W}^*$, $\mathcal{M}$ and $\mathcal{R}$ satisfies the property that their high-degree vertices induce a low-degree subgraph. 
In particular, for each graph in $\mathcal{W}^*$, $\mathcal{M}$ and $\mathcal{R}$, there is no ternary tree of height 2 as a subgraph. Let $\mathcal{H}\in \{\mathcal{W}^*,\mathcal{M},\mathcal{R}\}$. Then, Lemma \ref{lem2} implies that $m_{\mathcal{H}}(d)$ is at most $(1+o(1))\frac{\log d}{\log\log d}$. 
Furthermore, $m_{\mathcal{H}}(d) \geq m_\mathcal{S}(d)$ because any star forest $F$ of $\mathcal{S}$ is a subgraph of $\mathcal{H}$ of maximum degree $\Delta(F)$ for $d\geq 2$. Therefore, it follows from Theorem \ref{star} that $m_{\mathcal{H}}(d)\geq (1+o(1))\frac{\log d}{\log\log d}$. Hence, we obtain the following corollary.

\begin{cor}
    For $d \geq 2$, $m_{\mathcal{H}}(d)=(1+o(1))\frac{\log d}{\log \log d}$, where $\mathcal{H}\in \{\mathcal{W}^*,\mathcal{M},\mathcal{R}\}$.
\end{cor}
Since determining the value of $m(d)$ can be quite challenging even when $d$ is small, it is highly meaningful to further investigate cooperative colorings of special classes of graphs for small values of $d$. Therefore, we propose the following problem.
\begin{problem}
Determine the precise values of $m_\mathcal{G}(3)$ and $m_\mathcal{T}(4)$ where $\mathcal{G}$ represents the class of bipartite graphs and $\mathcal{T}$ represents the class of trees. 
\end{problem}

\section*{Acknowledgements}  
We are very grateful to the reviewers for their valuable feedback, which has significantly improved the quality of the paper.


\begin{thebibliography}{1}

\bibitem{ABCHJ} 
R. Aharoni, E. Berger, M. Chudnovsky, F. Havet, Z. Jiang, 
Cooperative colorings of trees and of bipartite graphs, 
{\it Electron. J. Combin.}
27 (2020) \#P1.41.

\bibitem{ABZ} 
R. Aharoni, E. Berger, R. Ziv,
Independent systems of representatives in weighted graphs, 
{\it Combinatorica} 
27 (2007) 253--267.

\bibitem{AHHS} 
R. Aharoni, R. Holzman, D. Howard, 
P. Spr\"{u}ssel, Cooperative colorings and independent systems of representatives,
{\it Electron. J. Combin.}
22 (2015) \#P2.27.

\bibitem{BM0}
J.A. Bondy, U.S.R. Murty,
Graph Theory,
Graduate Texts in Mathematics 244, Springer, 2008.

\bibitem{B}
P. Bradshaw, Cooperative colorings of forests, 
{\it Electron. J. Combin.}
30 (2023) \#P1.23. 

\bibitem{BM}
P. Bradshaw, T. Masa\v{r}\'{i}k, 
Single-conflict colorings of degenerate graphs, 
2021, arXiv:2112.06333.

\bibitem{EL}
P. Erd\H{o}s, L. Lov\'{a}sz,
Problems and results on 3-chromatic hypergraphs and some related questions,
{\it in} ``Infinite and
finite sets", (A. Hajnal {\it et al.}, Eds.), Colloq. Math.
Soc. J. Bolyai, Vol. 11. pp. 609--627, 
North-Holland, Amsterdam, 1975.

\bibitem{EMZ}
L. Esperet, M. Montassier, X. Zhu,
Adapted list coloring of planar graphs,
{\it J. Graph Theory}
62 (2009) 127--138.

\bibitem{H}
P. E. Haxell, A note on vertex list colouring, 
{\it Combin. Probab. Comput.}
10 (2001) 345--347.

\bibitem{HZ}
P. Hell, X. Zhu, On the adaptable chromatic number of graphs, 
{\it Eur. J. Comb.} 
29 (2008) 912--921.

\bibitem{KZ}
A.V. Kostochka, X. Zhu, 
Adapted list coloring of graphs and hypergraphs,
{\it SIAM J. Discrete Math.}
22 (2008) 398--408. 

\bibitem{LS}
Po-Shen Loh, B. Sudakov, Independent transversals in locally sparse graphs, 
{\it J. Combin. Theory Ser. B} 
97 (2007) 904--918.

\bibitem{LZ}
Y. Li, W. Zang, 
The independence number of graphs with a forbidden cycle and Ramsey numbers, 
{\it J. Comb. Optim.}
7 (2003) 53--359.

\bibitem{M}
M. Molloy,
The adaptable chromatic number and the chromatic number,
{\it J. Graph Theory}
 84 (2017) 53--56.

\bibitem{MB}
M. Molloy, B. Reed,
Graph Coloring and the Probabilistic Method, volume 23 of Algorithms and Combinatorics, Springer-Verlag, Berlin, 2002.

\bibitem{MT}
M. Molloy, G. Thron,  
An asymptotically tight bound on the adaptable chromatic number, 
{\it J. Graph Theory} 
71 (2012) 331--351.

\end{thebibliography}
\end{document}